\documentclass[11pt]{article}
\usepackage{latexsym}
\usepackage{graphicx}

\usepackage{latexsym}
\usepackage{graphicx}
\usepackage{amssymb}
\usepackage{amsmath}

\newtheorem{definition}{Definition}
\newtheorem{lemma}{Lemma}
\newtheorem{theorem}{Theorem}

\newenvironment{proof}{{\bf Proof.}}{\hfill $\Box$ \bigskip}
\usepackage{stmaryrd} 

\title{On the Behavior of Solutions of a System of Difference Equations}
 \author{ S. Basu and O. Merino\\
\small\scshape University of Rhode Island \\
{\it sukanya@math.uri.edu \quad merino@math.uri.edu }}
\date{\today}

\begin{document}

\maketitle

\begin{abstract}
We establish the relation between local stability of equilibria and slopes of critical curves for a specific class of difference equations. We then use this result to give global behavior results for nonnegative solutions of the system of difference equations
\begin{equation*}
\begin{array}{rcl}
x_{n+1} & = & \displaystyle \frac{b_1\, x_n}{1+x_n+c_1\, y_{n}} +h_1\\ \\
y_{n+1} & = & \displaystyle \frac{b_2\, y_n}{1+y_n+c_2\, x_{n}} +h_2
\end{array} 
\quad  n=0,1,\ldots, \quad (x_0,y_0) \in [0,\infty)\times [0,\infty)
\end{equation*}
with positive parameters. 
In particular, we show that the system has between one and three equilibria, and that
the number of equilibria determines global behavior as follows:
if there is only one equilibrium, then it is globally asymptotically stable.
If there are two equilibria, then one is a local attractor and the other one is nonhyperbolic.
If there are three equilibria, then they are linearly ordered in the south-east ordering of the plane,
and consist of a local attractor, a saddle point, and another local attractor.
Finally, we give sufficient conditions for having a unique equilibrium.
\\


\end{abstract}

\tableofcontents

\newpage

%
\section{Introduction} 
\label{sec: Introduction}
The study of dynamics of difference equations often requires that
equilibria be calculated first, followed by a local stability analysis
of the equilibria.  This is then complemented by other considerations
(existence of periodic points, chaotic orbits, etc.).
If the analysis is applied to a class of equations dependent on
one or more real parameters, the task is complicated by the fact 
that a formula is not always available for equilibria, and even if it is,
determination of stability properties of parameter-dependent
equilibria may be a daunting task.

In this paper we address the study of stability character of equilibria
 for  a parameter dependent class of systems of difference equations.
For this we develop local geometric criteria for planar systems of difference equations,
which under certain conditions determine local stability character of an equilibrium point.
This idea is used to study a class of difference equations in the plane,
for which we show that 
 the {\it number of equilibria
determines,  essentially in a unique way, the local stability character
of the equilibria, as well as  the kind of global dynamics that is possible}.

The class of difference equations we will focus on has its roots in biology applications.
The Leslie-Gower Model in difference equations is the two-species competition model 
\begin{equation*}
\tag{LG}
\begin{array}{rcl}
x_{n+1} & = & \displaystyle \frac{b_1\, x_n}{1+ c_{11}\,x_n+c_{12}\, y_{n}} \\ \\
y_{n+1} & = & \displaystyle \frac{b_2\, y_n}{1+c_{21}\,y_n+c_{22}\, x_{n}}\,.
\end{array} 
\quad n=0,1,\ldots, \quad (x_0,y_0) \in [0,\infty)\times [0,\infty)
\end{equation*}
It is a modified version of the Beverton-Holt equation
which is a well-known difference equation in Population Biology. The normalized Leslie-Gower Model
 \begin{equation*}
\tag{NLG}
\begin{array}{rcl}
x_{n+1} & = & \displaystyle \frac{b_1\, x_n}{1+x_n+c_1\, y_{n}}\\ \\
y_{n+1} & = & \displaystyle \frac{b_2\, y_n}{1+y_n+c_2\, x_{n}}\,,
\end{array} 
\quad n=0,1,\ldots, \quad (x_0,y_0) \in [0,\infty)\times [0,\infty)
\end{equation*}
was studied in detail by Liu and Elaydi \cite{LE}, J. Cushing et. al \cite{JMC}, and Kulenovi\'c and Merino \cite{KM}. 

In this paper we consider the system of equations
obtained by adding positive constants 
to the right-hand-side of the Leslie-Gower equations:
\begin{equation*}
\tag{LGI}
\begin{array}{rcl}
x_{n+1} & = & \displaystyle \frac{b_1\, x_n}{1+ c_{11}\,x_n+c_{12}\, y_{n}} + H_1\\ \\
y_{n+1} & = & \displaystyle \frac{b_2\, y_n}{1+c_{21}\,x_n+c_{22}\, y_{n}} + H_2\,.
\end{array} 
\quad n=0,1,\ldots, \quad (x_0,y_0) \in [0,\infty)\times [0,\infty)
\end{equation*}
The positive constants $H_1$ and $H_2$ in (LGI) may account for immigration.
The change of variables  $x_n = \frac{1}{a_{11}}\, \tilde{x}_n$, $y_n = \frac{1}{a_{22}}\, \tilde{y}_n$ 
normalizes (LGI) as follows:
\begin{equation*}
\tag{LGIN}
\begin{array}{rcl}
x_{n+1} & = & \displaystyle \frac{b_1\, x_n}{1+x_n+c_1\, y_{n}} +h_1\\ \\
y_{n+1} & = & \displaystyle \frac{b_2\, y_n}{1+y_n+c_2\, x_{n}} +h_2
\end{array} 
\quad n=0,1,\ldots, \quad (x_0,y_0) \in [0,\infty)\times [0,\infty)
\end{equation*}
where in (LGIN) the tildes have been removed from the variables to simplify notation, and 
$$
c_1 = \frac{c_{12}}{c_{22}}\, , \quad 
h_1 = c_{11}\, H_1\, ,  \quad 
c_2 = \frac{c_{21}}{c_{11}}\, ,  \quad \mbox{and} \quad 
h_2 = c_{22}\, H_1  \, .
$$
 \medskip
 
The system (LGIN) is an example of a {\it competitive system}, which is defined next.
Let $I$ and $J$ be intervals of real numbers and let 
$f:I\rightarrow I$ and $g:J\rightarrow J$ 
be continuous functions.  Consider the system
\begin{equation}
\label{eq: competitive}
\begin{array}{rcl}
x_{n+1} & = &  f(x_n,y_n) \\
y_{n+1} & = & g(x_n,y_n)
\end{array}
\,
\quad n=0,1,2,\ldots, \quad (x_0,y_0) \in I\times J
\end{equation}
System (\ref{eq: competitive}) is {\it competitive} if 
$f(x,y)$ is non-decreasing in $x$ and non-increasing in $y$, and 
$g(x,y)$ is non-increasing in $y$ and non-decreasing in $x$.
The map $T(x,y) := (f(x,y),g(x,y))$ associated to a 
competitive system (\ref{eq: competitive}) is
said to be {\it competitive}.
System (\ref{eq: competitive}) is {\it strongly competitive} if 
it is competitive, with strict coordinate-wise monotonicity
of the functions $f(x,y)$ and $g(x,y)$.
If $T$ is differentiable on an open set $R$, a sufficient condition for $T$ to be 
strongly competitive on $R$ is that the Jacobian matrix of $T$ at any $(x,y)\in R$
has nonzero entries with sign configuration
$$\left(\begin{array}{cc}+ & - \\ - & +\end{array}\right).
$$
Competitive systems of the form (\ref{eq: competitive}) 
have been studied by many authors  \cite{CKS},  \cite{HS1},   \cite{KN3}, \cite{HLS} and others. 
The term  {\it competitive}   was introduced by Hirsch \cite{Hirsch1982}
(see also \cite{HaleSomolinos1984})
for systems of autonomous differential equations 
$x_i^\prime = F_i(x_1,\ldots,x_n,t),$, $i=1,\ldots,n$
satisfying
$\frac{\partial F_i}{\partial x_j} \leq 0,\ i\neq j$. 
The main motivation for the study of these systems 
is the fact that many mathematical models in biological sciences
may be classified as   competitive (or cooperative)
\cite{deMS},  \cite{Smale}, 
\cite{LeonardMay75}.  
Consideration of Poincar\'e maps of these systems leads to
the concept of competitive and cooperative in the discrete case.

Denote with $\preceq_{se}$ the {\em South-East} 
partial order in the plane
whose nonnegative cone is  the standard fourth quadrant
$\{ (x,y) : x \geq 0,\ y \leq 0 \}$, that is,
$(x_{1},\,y_{1}) \,  \preceq_{se} \,  (x_{2},\,y_{2})$ if and only if $x_{1} \leq x_{2}$ and $y_{1} \geq y_{2}$.  
Competitive maps $T$ in the plane preserve the South-East ordering:
$T(x) \preceq_{se} T(y)$ whenever $x \preceq_{se} y$.
The concept of competitive (for maps) 
may be defined in terms of the 
order preserving properties of the maps.
Finally we note that the terms {\it equilibrium} (of a planar system of difference equations)
and {\it fixed point} (of a map) are used interchangeably in this paper.

\bigskip
 
This paper is organized as follows. 
In Section \ref{sec: A result} 
a relation between local stability of equilibria and slopes of critical curves for a specific class of difference equations is established. 
In Section \ref{sec: Preliminaries} 
it is shown that every solution to Eq.(LGIN) converges to an equilibrium.
Section \ref{sec: Main results}  contains the main result of this paper.
It states that Eq.(LGIN) has between one and three equilibria, and that
the number of equilibria determines global behavior as follows:
if there is only one equilibrium, then it is globally asymptotically stable.
If there are two equilibria, then one is a local attractor and the other one is nonhyperbolic.
If there are three equilibria, then they are linearly ordered in the south-east ordering of the plane,
and consist of a local attractor, a saddle point, and another local attractor.
Finally, in Section \ref{sec: A global attractivity result} 
we give sufficient conditions for Eq.(LGIN) to have a unique equilibrium.

\section{A preliminary result on critical sets}
\label{sec: A result}
Our first result establishes a connection between local stability of the fixed point of a planar map
and the slopes of certain curves at the fixed point.
The result will be useful in the proof of the main result in Section \ref{sec: Main results}.
\begin{theorem}
\label{thm : main theorem}
Let $R$ be a subset of $\mathbb{R}^2$ with nonempty interior, 
and let $T=(f,g):R\rightarrow R$ be a map of class $C^p$ for some $p\geq 1$.
Suppose that $T$ has a fixed point
$(\overline{x},\overline{y})\in {\rm int}\, R$ such that 
$$
a:= f_{x}(\overline{x},\overline{y}), \quad
b:= f_{y}(\overline{x},\overline{y}), \quad 
c:= g_{x}(\overline{x},\overline{y}), \quad
d:= g_{y} (\overline{x},\overline{y})
$$
satisfy
\begin{equation}
\label{ineq: hypothesis for thm}
0<a<1,\quad b\,c > 0,\quad 0<d<1,\quad
1 + (a+d) + a\, d-b\, c > 0
\end{equation}
Let $C_1$, $C_2$ be the {\em critical sets} 
$$
C_{1} :=\{ (x,\,y) : x = f\,(x,\,y) \}
\quad \mbox{and} \quad 
C_{2} :=\{ (x,\,y) : y = g\,(x,\,y) \} 
$$
Then,
\begin{itemize}
\item[\rm i.]
There exists neighborhood  $I \subset \mathbb{R}$ of $\overline{x}$ 
and $J \subset \mathbb{R}$ of $\overline{y}$ such that the sets
$C_{1}\cap (I \times J)$ and $C_{2} \cap (I \times J)$ are the graphs
of class $C^p$ functions $y_1(x)$ and $y_2(x)$ for $x \in I$. 
\item[\rm ii.]
The eigenvalues $\lambda_{1}$, $\lambda_{2}$ of 
the jacobian matrix of $T$ at $(\overline{x},\overline{y})$ are real, distinct, and the 
eigenvalue with larger absolute value is positive.
 If $\lambda_{1}$ is the larger eigenvalue,  then 
\begin{equation}
\label{eq: signs of derivatives}
-1< \lambda_2 < 1 \quad\quad \mbox{\rm and} 
\quad\quad \mbox{\rm sign}\,(\, \lambda_1-1\, ) = 
\left\{
\begin{array}{rcl}
- \,\mbox{\rm sign}\,(\, y_1^\prime(\overline{x}) -y_2^\prime(\overline{x}))\,\,\, \mbox{ if $b < 0$}\\ \\
\mbox{\rm sign}\,(\, y_1^\prime(\overline{x}) -y_2^\prime(\overline{x}))\,\,\, \mbox{ if $b > 0$}
\end{array}
\right.
\end{equation}
\end{itemize}
\end{theorem}
%
\begin{proof}
\begin{itemize}
\item[\rm i.]
The existence of of $I$ and $J$ and of 
smooth functions $y_{1}(x)$ and $y_{2}(x)$ defined in $I$ 
as in the statement of the Theorem 
 is guaranteed by the hypotheses and the  Implicit Function Theorem.  Moreover, 
  \begin{equation}
 \label{eq: y1' and y2'}
y_{1}^\prime(x) = \displaystyle\frac{1-f_{x}(x,\,y)}{f_{y}(x,\,y)} 
\quad \mbox{and} \quad 
y_{2}^\prime(x) = \displaystyle\frac{g_{x}(x,\,y)}{1-g_{y}(x,\,y)}\, , 
\quad x \in I
 \end{equation}

\item[\rm ii.]
The characteristic polynomial of the Jacobian of $T$,
$$
p(\lambda) = \lambda^2 -(a+d) \, \lambda  + (a\,d-b\,c)
$$
has positive discriminant, thus its roots $\lambda_1,\ \lambda_2$ are real.
Since $\lambda_1+\lambda_2 = a+d > 0$, 
then the root $\lambda_1$ with larger absolute value is positive, and $|\lambda_2| < \lambda_1$.
Now by the hypothesis $a, \, d\in(0,1)$ it follows that $\lambda_1+\lambda_2 = a+d <2$,
which implies $\lambda_2 < 1$. Since $p(\lambda)$ has at least one positive root and since
by hypothesis $p(-1) = 1 + (a+d) + a\, d-b\, c > 0$, we have $-1 < \lambda_2$.
To prove the second part of (\ref{eq: signs of derivatives}), 
note that from \eqref{eq: y1' and y2'}, we have 
\begin{equation}
\label{eq: y1' - y2'}
\begin{array}{lcl}
y_1^\prime(\overline{x}) -y_2^\prime(\overline{x}) 
&=& \displaystyle\frac{1- a}{ b} - \displaystyle\frac{c}{1- d}
  =  - \displaystyle\frac{1- (a + d) + a\,d - b\,c}{b\, (1 - d)} \\ \\
&=& - \displaystyle\frac{p\,(1)}{b\,(1 - d)} 
 =    \displaystyle\frac{(\lambda_{1}-1)\,(1- \lambda_{2})}{b\,(1 - d)}
\end{array}
\end{equation}
The result is a direct consequence of   \eqref{eq: y1' - y2'},  the inequality $ \lambda_{2} < 1$ 
and  hypotheses \eqref{ineq: hypothesis for thm}.
\end{itemize}
\end{proof}
\section{Every solution converges to an equilibrium}
\label{sec: Preliminaries}
In this section we show that every solution to Eq.(LGIN) 
converges to an equilibrium.
Let   
$$
f(x,y) =  \frac{b_1\, x}{1+x+c_1\, y} +h_1, 
\quad \mbox{and} \quad 
g(x,y) =  \frac{b_2\, y}{1+y+c_2\, x} +h_2 \, .
$$
Then the map  $T(x,y) = (f(x,y),g(x,y))$ associated with  {\rm (LGIN)} is
\begin{equation}
\label{eq: map}
T\,(x,\,y) = \left( \displaystyle \frac{b_1\, x}{1+x+c_1\, y} +h_1,\,
 \displaystyle \frac{b_2\, y}{1+y+c_2\, x} +h_2 \right),\quad (x,\,y) \in [0,\infty)\times[0,\infty)
\end{equation}
For future reference we give the jacobian matrix of $T$ at $(x,y)$:
\begin{equation}
\label{eq: Jacobian}
J_{T} (x,\,y) = \left(
\begin{array}{cc} 
\displaystyle\frac{b_{1}\, (c_{1} \,y\,+\,1)}{(x\,+\,c_{1}\, y\,+\,1)^2}   & -\displaystyle\frac{b_{1}\, c_{1} \,x}{(x\,+\,c_{1} \,y\,+\,1)^2} \\ \\
 -\displaystyle\frac{b_{2}\, c_{2}\, y}{(c_{2} \,x\,+\,y\,+\,1)^2}   & \displaystyle\frac{b_{2}\, (c_{2} \,x\,+\,1)}{(c_{2}\, x\,+\,y\,+\,1)^2}
\end{array}
\right)
\end{equation}
By direct inspection of (\ref{eq: Jacobian}) we obtain the following result.
\begin{lemma}
\label{lemma: T is competitive}
The system of difference equations {\rm (LGIN)} is strongly competitive
on $[0,\infty)\times[0,\infty)$.
\end{lemma}
The following lemma has an easy proof which we skip.
\begin{lemma}
\label{lemma: T is bounded}
$T(\, [0,\infty)\times[0,\infty) \, )  \subset [h_{1},\,h_{1}\,+\,b_{1} ] \times [h_{2},\,h_{2}\,+\,b_{2} ]$.
\label{lem : boundedness}
\end{lemma}
\begin{theorem}
\label{thm: convergence}
Every solution of {\rm Eq.(LGIN)} converges to an equilibrium.
\end{theorem}
\begin{proof}
It is easy matter to show that the map $T$ is one-to-one. 
From (\ref{eq: Jacobian}) the determinant of $J_T(x,y)$ is 
\begin{equation}
\label{eq : det of Jac}
\mbox{det}\,(Jac_{T}(x,\,y)) = \displaystyle\frac{b_1\,b_2\, (c_2\, x+c_1\,  
   y+1)}{(c_2\, x+y+1)^2\, (x+c_1\, y+1)^2} \,,
\end{equation}
which is clearly positive for $(x,\,y) \in [0,\infty)\times[0,\infty)$.
If follows that the map $T$ satisfies hypothesis $(H+)$ in \cite{HLS}.
By Lemma 4.3 of \cite{HLS} and Theorem 4.2 of \cite{HLS} we have that
every solution of  Eq.(LGIN)  is 
eventually coordinate-wise monotone. 
Since by Lemma \ref{lem : boundedness} every solution 
enters the compact set 
$[h_{1},\,h_{1}\,+\,b_{1} ] \times [h_{2},\,h_{2}\,+\,b_{2} ]$, 
the conclusion follows.
\end{proof}


\section{Number of equilibria and global behavior}
\label{sec: Main results}

By solving for $y$ and $x$, respectively, 
in the equations defining the {\it critical sets} 
$C_{1}=\{ (x,\,y) \in {\mathbb{R}}^{2} : x = f\,(x,\,y) \} $ 
and 
$C_{2}=\{(x,\,y) \in {\mathbb{R}}^{2} :  y = g\,(x,\,y) \} $ 
and taking derivatives 
we obtain the following result.
\begin{lemma}
\label{lem: hyperbolas}
All branches of the sets
\begin{equation*}
\label{eq: critical curves system}
\begin{array}{rcl}
C_{1}   
&  =& \{(x,\,y) \in {\mathbb{R}}^{\,2} : x^2 + c_1 \, x \, y +(1-b_1-h_1)\, x - c_1\,h_1\,y-h_1 =  0 \}  \\ \\
C_{2}  
&  =& \{ (x,\,y) \in {\mathbb{R}}^{2} : y^2 + c_2 \, x \, y +(1-b_2-h_2)\, y- c_2\,h_2\,x-h_2 =  0\}  
\end{array}
\end{equation*}
are the graphs of decreasing functions of one variable. 
\end{lemma}
 
\begin{lemma}
\label{lem : hyp. oh Thm. 1}
 The map $T$ satisfies hypotheses \eqref{ineq: hypothesis for thm} of Theorem \ref{thm : main theorem}.
\end{lemma}
\begin{proof}
Set 
$a:= f_{x}(\overline{x},\overline{y})$,  
$b:= f_{y}(\overline{x},\overline{y})$,   
$c:= g_{x}(\overline{x},\overline{y})$, 
$d:= g_{y} (\overline{x},\overline{y})$.
Implicit differentiation of  the equations defining   
$C_{1}$ and $C_{2}$ in (\ref{eq: critical curves system}) 
at  $(\overline{x},\,\overline{y})$ gives
\begin{equation}
\label{eq : hypotheses of Thm. 1}
y_1^\prime(\overline{x}) =  \displaystyle\frac{1- a}{ b}  \quad\mbox{and} \quad y_2^\prime(\overline{x}) 
 = \displaystyle\frac{c}{1- d}
\end{equation}
 From Lemma \ref{lem: hyperbolas}, $y_1^\prime(\overline{x})<0$ and $y_2^\prime(\overline{x})<0$,
 and from (\ref{eq: Jacobian}), $b<0$ and $c<0$.
 It follows that $a < 1$ and $d < 1$ in \eqref{eq : hypotheses of Thm. 1}. 
 Furthermore, since the map $T$ is strongly competitive, 
 $a>0$, $d>0$, $b<0$, $c<0$. 
 Hence,  $T$ must satisfy inequalities $0 < a < 1, \,\,  0 < d < 1$ and $ b\,c > 0$
from hypotheses  \eqref{ineq: hypothesis for thm} of Theorem \ref{thm : main theorem}.
 Finally note that ${\rm det} (J_{T} (\overline{x},\,\overline{y})) = a\,d - b\,c > 0$ 
 by \eqref{eq : det of Jac}, hence $1 + (a+d) + a\, d-b\, c > 0$. 
\end{proof}

Denote with $Q_\ell(a,b)$, $\ell=1,2,3,4$ the four regions
$Q_1(a,b) := \{ (x,y) \in \mathbb{R}^2\, : \ a \leq x,\ b\leq y \ \}$,  
$Q_2(a,b) := \{ (x,y) \in \mathbb{R}^2\, : \ x \leq a,\ b\leq y \ \}$, 
$Q_3(a,b) := \{ (x,y) \in \mathbb{R}^2\, : \ x \leq a,\ y\leq b \ \}$, 
$Q_4(a,b) := \{ (x,y) \in \mathbb{R}^2\, : \ a \leq x,\ y\leq b \ \}$. 
\begin{lemma}
\label{lem : fixed pts. in Quads. I and III}
 Each of the sets $Q_{1}\,(h_{1},\,h_{2})$ and  $Q_{3}\,(h_{1},\,h_{2})$ contain at least one equilibrium of  Eq.(LGIN).
\end{lemma}
\begin{proof}
The existence of an equilibrium of Eq.(LGIN) in the invariant and attracting set  
 $B:=[h_{1},\,h_{1}\,+\,b_{1} ] \times [h_{2},\,h_{2}\,+\,b_{2} ] $ from Lemma \ref{lem : boundedness}  
 is guaranteed by the Schauder Fixed Point Theorem \cite{RBH} for continuous maps on compact, 
 convex and invariant sets. Since $B \subset Q_{1} (h_{1},\,h_{2})$, such equilibrium points are 
 elements of $Q_{1} (h_{1},\,h_{2})$.
To see that $Q_{3}\,(h_{1},\,h_{2})$ contains an equilibrium of Eq.(LGIN), 
note that the critical curves $C_{1}$ and $C_{2}$ can be given explicitly as functions of $x$ : 
$$
C_{1}: \quad y_{1}(x) = \displaystyle\frac{x^2 +(1-b_{1}\, -h_{1} )\,x -h_{1}}{c_{1} (h_{1}-x)}\,,\,\,x \neq h_{1} 
$$
$$
C_{2}:
\left\{
\begin{array}{rcl}
 y_{2+}(x) &= &\displaystyle\frac{1}{2} \,\left(- 1 + b_{2} + h_{2} - c_{2}\,
   x + \sqrt{(-\,b_{2} - h_{2} + c_{2}\, x +1)^{2}  + 4\,( c_{2} \,x\, h_{2} + h_{2})}\right)\\ \\
  y_{2-}(x) &= & \displaystyle\frac{1}{2} \,\left(- 1 + b_{2} + h_{2} - c_{2}\,
   x - \sqrt{(-\,b_{2} - h_{2} + c_{2}\, x +1)^{2}  + 4\,( c_{2} \,x\, h_{2} + h_{2})}\right)
\end{array}
\right. 
$$
Then,
\begin{equation}
\label{eq: limits}
\lim_{x \rightarrow -\infty}\,\, y_{1}(x) - y_{2 -}(x)  = \infty \quad \mbox{and} \quad \lim_{x \rightarrow h_{1}}\,\, y_{1}(x) - y_{2-}(x)= -\infty\,.
\end{equation}
One can conclude from \eqref{eq: limits} and the continuity of $y_{1}(x)$, $y_{2 -}(x) $ 
that there exists $c < h_1$ such that $y_{1}(c) = y_{2-}(c) $. 
Since $x = h_{1}$ is a vertical asymptote of $C_{1}$ 
and $y = h_{2}$ is a horizontal asymptote of $C_{2}$, 
it follows from the decreasing characters of $y_{1}(x)$ and $y_{2-}(x)$ that $(c,\,y_{1}(c) )$ must lie in $Q_{3}\,(h_{1},\,h_{2})$. 
\end{proof}                       

\begin{definition} 
\label{def : multiplicity}
Let $k$ be a positive integer. Let $(x_{0},\,y_{0})\in \mathbb{R}^{2}$ be an intersection point of the graphs $C_{1}$ and $C_{2}$ of two $k$-times differentiable functions $f_{1},\,f_{2}$ defined in a neighborhood of $x_{0}$. The point $(x_{0},\,y_{0})$ is a 
{\em contact point of $C_{1}$ and $C_{2}$ of order $k$} if
$$
\left\{
\begin{array}{rcl}
f_{1}^{(\ell)}(x_{0}) &=&f_{2}^{(\ell)} (x_{0})\,,\,\, 0 \leq \ell < k\\ \\
f_{1}^{(k)}(x_{0}) &\neq&f_{2}^{(k)} (x_{0})
\end{array}
\right.
$$
\end{definition}
Note that $C_{1}$ and $C_{2}$ intersect transversally at 
$(x_{0},\,y_{0})$ if and only if $(x_{0},\,y_{0})$ 
is a contact point of $C_{1}$, $C_{2}$ of order one.
\begin{lemma}
\label{lem : contact pt. - multiplicity}
Let $P_0=(x_0,y_0)$ be an intersection point of two
non-singular algebraic curves 
$C_1$ and $C_2$ in the plane that have
 no common component through $P_0$.
 If $C_1$ and $C_2$ have tangent lines at $P_0$ that are not parallel to either axis,
then there exists a neighborhood $U$ of $P_0$ for which both
$U \cap C_1$ and $U \cap C_2$ are the graphs of real analytic strictly
monotonic functions $y_1(x)$ and $y_2(x)$ for $x$ in some neighborhood
of $x_0$, and   the order of 
$P_0$ as a contact point of $C_1$ and $C_2$
equals the multiplicity of $P_0$ as an
intersection point of the algebraic curves $C_1$ and $C_2$.
\end{lemma}
\begin{proof}
Without loss of generality we may assume $(x_0,y_0) = (0,0)$.
Let $y_1(t)$ and $y_2(t)$ be such that 
$U\cap C_1$ and $U\cap C_2$ are parametrized by 
$(t,y_1(t))$ and $(t,y_2(t))$ for $t \in I$.
Then $y_1(t)$ and $y_2(t)$ are real analytic on $I$.
The rest of the statement follows.
\end{proof}
\begin{lemma}
\label{lem : d_x = d_y =k}
 Let $p$ and $q$ be real analytic strictly monotone
 functions defined in neighborhoods of $x_0$ and $y_0$
 respectively, such that $p(x_0)=y_0$,
 $q(y_0)=x_0$, and neither $p\circ q$ nor $q \circ p$
 is the identity function.
 Let $k$ be the order of $(x_0,y_0)$ as a contact point of
 $C_1:= \{(x,y):y=p(x)\}$ and  $C_2:= \{(x,y):x=q(y)\}$,
 and let $d_y$ and $d_x$ be the multiplicities of $y_0$
 and $x_0$ as zeros of $y-p(q(y))$ and $x-q(p(x))$ respectively.
 Then $d_y = d_x=k$.
 \end{lemma}
\begin{proof}
Consider the function $\phi(y):= y - p (\widetilde{p}(y)) $, where $\widetilde{p}$ is the inverse of the strictly monotone function $p$. Clearly, $\phi(y) = 0$ in a neighborhood of $y_{0}$. It follows that  $\phi ^{\ell}(y_{0)}=0$ for $\ell > 0$. In particular, $\phi ^{\prime}(y_{0})=0$ from which we get 
\begin{equation}
\label{eq: p'}
\widetilde{p}\,^\prime(y_{0}) = \displaystyle\frac{1}{p\,^\prime(\widetilde{p}(y_{0}))} = \displaystyle\frac{1}{p\,^\prime(x_{0})}\,.
\end{equation}
Now consider the function $\psi(y)= y-p(q(y))$. Since $y_{0}$ is a root of $\psi(y)$ of multiplicity $d_{y}$ by hypothesis, we must have $\psi ^{(\ell)}(y_{0)}=0$ for $0 < \ell < d_{y}$ and $\psi ^{(d_{y})}(y_{0)}\neq0$. In particular, $\psi ^{\prime}(y_{0}) = 0$ from which we have
\begin{equation}
\label{eq: q'}
q\,^\prime(y_{0}) = \displaystyle\frac{1}{p\,^\prime(q(y_{0}))} = \displaystyle\frac{1}{p\,^\prime(x_{0})}\,.
\end{equation}
 From \eqref{eq: p'} and \eqref{eq: q'}, it follows that $\widetilde{p}\,^\prime(y_{0}) = q\,^\prime(y_{0}) $. Similarly, one can show that $\widetilde{p}\,^{(\ell)}(y_{0}) = q\,^{(\ell)}(y_{0}),\, 2 \leq \ell < d_{y} $. However,  since $\psi ^{(d_{y})}(y_{0})\neq0$, we must have $\widetilde{p}\,^{(d_{y})}(y_{0}) \neq q\,^{(d_{y})}(y_{0})$. Hence
\begin{equation}
\label{eq : intersection pt. - derivative rel. ,dx, dy etc.}
\left\{
\begin{array}{rcl}
\widetilde{p}\,^{(\ell)}(y_{0}) &=& q\,^{(\ell)} (y_{0})\,,\,\, 0 \leq \ell < d_{y} \\ \\
\widetilde{p}\,^{(d_{y})}(y_{0}) &\neq& q\,^{(d_{y})} (y_{0})
\end{array}
\right.
\end{equation}
It is a direct consequence of Definition \ref{def : multiplicity} and \eqref{eq : intersection pt. - derivative rel. ,dx, dy etc.} that $d_{y} = k$. By a similar argument using the equation defining $C_{2}$, one can show that $d_{x} = k$.
\end{proof}

\noindent  The main result of this paper is the following.
\begin{theorem}
\label{th: main result 1}
The following statements are true:
\begin{itemize}
\item[\rm (i.)]
{\rm Eq.(LGIN)} has at least one and at most three 
equilibria in $[0,\infty)^2$.  The set of equilibrium points in $[0,\infty)^2$
is  linearly ordered by $\preceq_{se}$.
\item[\rm (ii.)]
If {\rm Eq.(LGIN)} has exactly one equilibrium in $[0,\infty)^2$, 
then it is is globally asymptotically stable.
\item[\rm (iii.)]
If {\rm Eq.(LGIN)} has  three distinct equilibria in 
$[0,\infty)^2$, say $(\overline{x}_{\ell},\, \overline{y}_{\ell}),\,
l=1, ...,3$, with \\ 
$(\overline{x}_1,\overline{y}_1) \preceq_{se} (\overline{x}_2,\overline{y}_2)
\preceq_{se}(\overline{x}_3,\overline{y}_3)\, ,$
then $(\overline{x}_1,\overline{y}_1)$  and $(\overline{x}_3,\overline{y}_3)$ 
are locally asymptotically stable, while
$(\overline{x}_2,\overline{y}_2)$ is a saddle point.
The global stable manifold of $(\overline{x}_2,\overline{y}_2)$
is the graph of a continuous increasing function of the first variable
with endpoints in the boundary of $[0,\infty)^2$, which is a 
separatrix of the basins of attraction of 
$(\overline{x}_1,\overline{y}_1)$ and $(\overline{x}_3,\overline{y}_3)$.
\item[\rm (iv.)]
If there exist exactly two equilibria in $[0,\infty)^2$,
then one is locally asymptotically stable and the other is a 
nonhyperbolic fixed point.  If $(x_1,y_1)$ and $(x_2,y_2)$ are the two equilibria
with $(x_1,y_1)\preceq_{se}(x_2,y_2)$, then $Q_2(x_1,y_1)$ is a subset of the
basin of attraction of $(x_1,y_1)$, and 
$Q_4(x_2,y_2)$ is a subset of the
basin of attraction of $(x_2,y_2)$.
\end{itemize}
\end{theorem}
%
%
\begin{proof}
\begin{itemize}
\item[\rm (i.)]
It is a consequence of B{\'e}zout's Theorem (Theorem 3.1, Chapter III in \cite{RJW}) that the hyperbolas $C_{1}$ and $C_{2}$ given in (\ref{eq: critical curves system}) must intersect in at most four points. Since intersection points of $C_{1}$ and $C_{2}$ are precisely the equilibrium points of Eq.(LGIN), it follows that Eq.(LGIN) must have at most four equilibrium points.
 Lemmas \ref{lemma: T is bounded} and  \ref{lem : fixed pts. in Quads. I and III}  guarantee  
 that at most three of these equilibria can lie in the invariant attracting box 
 $B:=[h_1,h_2]\times[h_1+b_1,h_2+b_2]$ 
 and hence in $Q_{1}\,(h_{1},\,h_{2})$, which proves the first part of statement (i.).
To see that the set of equilibrium points in $[0,\infty)^2$ is  linearly ordered by $\preceq_{se}$ , 
note that the equilibria are precisely the intersection points of the decreasing critical curves 
$C_{1}$ and $C_{2}$ given in Lemma \ref{lem: hyperbolas}. 
\item[\rm (ii.)]
Let $T$ be the map of Eq.(LGIN), and let $(x,y) \in [0,\infty)^2$.
By Lemma \ref{lemma: T is bounded} and since $T$ is competitive,
 $(h_1,h_2+b_2) \preceq_{se} T(h_1,h_2+b_2) \preceq_{se} T(h_1+b_1,h_2) \preceq_{se} (h_1+b_1,h_2)$. 
By induction,
 $T^n(h_1,h_2+b_2) \preceq_{se}  T^{n+1}(h_1,h_2+b_2) \preceq_{se} T^{n+1}(h_1+b_1,h_2) \preceq_{se} T^n(h_1+b_1,h_2)$
 for $n=1,2,\ldots$. 
Then the sequences $\{T^n(h_1,h_2+b_2)\}$ and $\{T^n(h_1+b_1,h_2)\}$
are respectively monotonically increasing and decreasing, they are bounded and  converge
(this follows from the bounded and monotonic coordinate-wise character of the sequences).
Since both must converge to a fixed point, the limit point is the unique fixed point of $T$.  
Again by Lemma \ref{lemma: T is bounded}, for $(x,y) \in [0,\infty)^2$, $T(x,y) \in B$, hence
$(h_1,h_2+b_2) \preceq_{se} T(x,y) \preceq_{se} (h_1+b_1,h_2)$, and by induction,
$T^n(h_1,h_2+b_2) \preceq_{se} T^{n+1}(x,y) \preceq_{se} T^n(h_1+b_1,h_2)$ for $n=1,2,\ldots$.
Thus $T^n(x,y) \rightarrow (\overline{x},\overline{y})$ as $n\rightarrow \infty$.
This proves global attractivity of $(\overline{x},\overline{y})$
Stability follows from the fact that 
$T^n(h_1,h_2+b_2) \preceq_{se} T^{n+\ell}(x,y) \preceq_{se} T^n(h_1+b_1,h_2)$ for 
$\ell=1,2,\ldots$, $n=1,2,\ldots$, and $\mbox{\rm dist}(T^n(h_1,h_2+b_2),T^n(h_1+b_1,h_2)) \rightarrow 0$.
\item[\rm (iii.)]  
Suppose Eq.(LGIN) has three distinct equilibria in $[0,\infty)^2$, say $(\overline{x}_{\ell},\, \overline{y}_{\ell}),\,l=1, ...,3$, with $(\overline{x}_1,\overline{y}_1) \preceq_{se} (\overline{x}_2,\overline{y}_2)\preceq_{se}(\overline{x}_3,\overline{y}_3)$. 
Note that the equilibria are precisely the intersection points of the critical curves $C_{1}$ and $C_{2}$ given in (\ref{eq: critical curves system}). 
From Bezout's Theorem, Lemma \ref{lem : fixed pts. in Quads. I and III} and Lemma \ref{lem : contact pt. - multiplicity},
 each equilibrium in $[0,\infty)^2$ must be a contact point of $C_{1}$ and $C_{2}$ of order one. 
 It follows from the remark after Definition \ref{def : multiplicity} that  $C_{1}$ and $C_{2}$ 
 must intersect transversally at each of the three equilibria in $[0,\infty)^2$. 
 Furthermore, solving for y and x respectively in the equations defining $C_{1}$ and $C_{2}$ in $(\ref{eq: critical curves system}) $ 
 gives that the vertical asymptote of $C_{1}$ is $x = h_{1}$ and the horizontal asymptote of $C_{2}$ is $y = h_{2}$. 
 The asymptotes guarantee that in order to have three intersection points in $[0,\infty)^2$, 
 the slopes of the functions $y_{1}{(x)}$ and $y_{2}{(x)}$ of $C_{1}$ and $C_{2}$ respectively, 
 must satisfy the relations 
 $ y_1^\prime(\overline{x}_1) < y_2^\prime(\overline{x}_1) ,\,  
 y_1^\prime(\overline{x}_2) >y_2^\prime(\overline{x}_2)$ 
 and $y_1^\prime(\overline{x}_3) < y_2^\prime(\overline{x}_3)\,$. 
 Theorem \ref{thm : main theorem} then gives that  $(\overline{x}_1,\overline{y}_1)$  
 and $(\overline{x}_3,\overline{y}_3)$ must be locally asymptotically stable, 
 while $(\overline{x}_2,\overline{y}_2)$ must be a saddle point. 
 The rest of the proof follows from Theorem 4 and Corollary 1 in \cite{KM2}. 
%
%
\item[\rm (iv.)]  
Suppose Eq.(LGIN) has two distinct equilibria in $[0,\infty)^2$, say,  
$(\overline{x}_1,\overline{y}_1)$ and $ (\overline{x}_2,\overline{y}_2)$. 
Note that $(\overline{x}_1,\overline{y}_1)$ and $ (\overline{x}_2,\overline{y}_2)$ 
cannot both be contact points of $C_{1}$ and $C_{2}$ of order one. 
Indeed if they were, then as a result of the remark after 
Definition \ref{def : multiplicity}, $C_{1}$ and $C_{2}$ 
would have to intersect transversally at $(\overline{x}_1,\overline{y}_1)$ 
and $ (\overline{x}_2,\overline{y}_2)$. Thus by Theorem \ref{thm : main theorem}, 
both equilibria would have to be locally asymptotically stable. 
But this would lead to a contradiction since  by Theorem 4 of \cite{DH}, 
at least one of the equilibria has to be a non-attracting equilibrium. 
Suppose  $(\overline{x}_1,\overline{y}_1)$ and $ (\overline{x}_2,\overline{y}_2)$ 
are contact points of $C_{1}$ and $C_{2}$ of orders two and one respectively. 
Then by the remark after Definition \ref{def : multiplicity}, $C_{1}$ and $C_{2}$ 
must intersect tangentially at $(\overline{x}_1,\overline{y}_1)$ 
and transversally at $ (\overline{x}_2,\overline{y}_2)$. 
Hence by Theorem \ref{thm : main theorem} , $(\overline{x}_1,\overline{y}_1)$  
must be a non-hyperbolic fixed point and $(\overline{x}_2,\overline{y}_2)$ 
must be locally asymptotically stable. 
The statement on basins of attraction is a consequence of $T$ being competitive
and the fact that $Q_2(x_1,y_1)$ and $Q_4(x_2,y_2)$ are invariant sets for $T$.
Indeed, if $(x,y) \in Q_2(x_1,y_1)$ then $T^n(x,y) \in Q_2(x_1,y_1)$.
Since $\{T^n(x,y)\}$ converges to a fixed point by Theorem \ref{thm: convergence},
it must converge to the only fixed point in $Q_2(x_1,y_1)$, namely, $(x_1,y_1)$.
A similar argument applies if $(x,y) \in Q_4(x_2,y_2)$.
%
%
\end{itemize}
\end{proof}

\section{A global attractivity result}
\label{sec: A global attractivity result}
A version of the following theorem 
was proved first in \cite{KuLO}. 
It is given here for easy reference.
\begin{theorem} 
\label{thm : M-m Theorem }
Let $[a,b]$ and $[c,d]$ be intervals of real numbers and let 
$f:[a,b]\rightarrow [a,b]$ and $g:[c,d]\rightarrow [c,d]$ 
be continuous functions that satisfy the following properties {\rm (a)} and {\rm (b)}:
\begin{itemize}
\item[\rm (a)]
$f(x,y)$ is non-decreasing in $x$ and non-increasing in $y$, and 
$g(x,y)$ is non-increasing in $y$ and non-decreasing in $x$.
\item[\rm (b)]
For every $(m,M),\, (\overline{m},\overline{M}) \in [a,b]\times[c,d]$,
\begin{equation}
\label{eq: Mm system}
\begin{array}{rclrcl}
m & = & f(m,\overline{M})\, , \quad \overline{m} & = & f(\overline{m},M), \\
M & = & f(\overline{m},M) \, , \quad \overline{M} & =  & g(m,\overline{M})
\end{array}
\quad 
\Longrightarrow 
\quad 
\begin{array}{c}
m = \overline{m} \\
M = \overline{M}
\end{array}
\end{equation}
\end{itemize}
Then the system of difference equations 
\begin{equation}
\label{eq: system Mm}
\begin{array}{rcl}
x_{n+1} & = &  f(x_n,y_n) \\
y_{n+1} & = & g(x_n,y_n)
\end{array}
\,
\quad n=0,1,2,\ldots, \quad (x_0,y_0) \in [a,b]\times[c,d]
\end{equation}
has a unique equilibrium $(\overline{x},\overline{y})$ in $[a,b]\times[c,d]$,   
 and the unique equilibrium is a global attractor.
 \end{theorem}
%
%
%
%
Next we give a sufficient condition for the existence of a unique equilibrium.
\begin{theorem}
\label{th: main result 2}
If at least one of the following conditions is satisfied
\begin{equation}
\label{eq: condns. for a unique equilibrium}
\begin{array}{lcl}
{\rm (a)}\quad1-b_{1} \,+\,h_{1}\,+\,c_{1}\,h_{2} \geq 0 \quad  \mbox{and} \quad
 1-b_{2} \,+\,h_{2}\,+\,c_{2}\,h_{1} \geq 0 
\\ \\
{\rm (b)}\quad  c_1 \,c_2  \leq 1
\end{array}
\end{equation}
 then Eq.(LGIN) has 
a unique equilibrium in $[0,\infty)^2$.
\end{theorem}

\begin{proof}
\begin{itemize}
\item[\rm (a)]
Suppose 
\begin{equation}
\label{eq: hypothesis}
1-b_{1} +h_{1}+c_{1}\,h_{2} \geq 0 \quad  \mbox{and}  \quad 
 1-b_{2} +h_{2}+c_{2}\,h_{1} \geq 0.
 \end{equation}
From direct inspection of (\ref{eq: Jacobian})  one can see that the functions 
$f(x,y)$ and $g(x,y)$ in $T(x,y) = (f(x,y),g(x,y))$ satisfy hypothesis (a) of Theorem \ref{thm : M-m Theorem }. 
To verify hypothesis (b) of Theorem \ref{thm : M-m Theorem },
note that for $T$ as in (\ref{eq: map}), the system of equations in (\ref{eq: Mm system}) 
is given by
\begin{equation}
\label{eq: M-m Theorem 1}
\left.
\begin{array}{rcl}
m & = & \displaystyle \frac{b_1\, m}{1+m+c_1\, \overline{M}} +h_1\\ \\
M & = & \displaystyle \frac{b_1\, M}{1+M+c_1\, \overline{m}} +h_1
\end{array}
\right\}
\\ \\
\quad \mbox{and} \quad\quad
\\ \\
\left.
\begin{array}{rcl}
\overline{m} &= & \displaystyle \frac{b_2\, \overline{m}}{1+\overline{m}+c_2\, M} +h_2\\ \\
\overline{M} &= & \displaystyle \frac{b_2\, \overline{M}}{1+\overline{M}+c_2\, m} +h_2
\end{array}
\right\}
\end{equation}
Algebraic manipulation of the equations in (\ref{eq: M-m Theorem 1}) yields the equation
\quad \quad 
\begin{equation}
\label{eq: M-m Theorem 2}
\begin{array}{lcl}
c_{2}\,(M\,-\,m)\,[(m\,-\,h_{1}) \,+\,(M\,-\,h_{1})\,+\,(1\,-\,b_{1}\,+\,h_{1}\,+\,c_{1}\,h_{2})]  \\ 
\,+\,c_{1}(\overline{M}\,-\,\overline{m})\,[(\overline{m}\,-\,h_{2})\, +\, (\overline{M}\,-\,h_{2})\,+\,(1\,-\,b_{2} \,+\,h_{2}\,+\,c_{2}\,h_{1})] =0
\end{array}
\end{equation}
Using hypothesis (\ref{eq: hypothesis}) and the facts $m,\, M \geq h_{1} $\, and \,$\overline{m},\, \overline{M} \geq h_{2} $, we obtain that in equation (\ref{eq: M-m Theorem 2}), $m = M$ and $\overline{m} = \overline{M}$ which shows that
hypothesis (b) holds.  Then  Theorem \ref{thm : M-m Theorem } implies that there is a unique equilibrium.

\bigskip

 \item[\rm (b)]
 Suppose  $ c_1 \,c_2  \leq 1$.
 Applying the transformation $x = X + h_{1},\,y = Y + h_{2}$ 
 to the equations defining the sets $C_{1}$, $C_{2}$ in
Lemma \ref{eq: critical curves system} 
  gives rise to the equations
\begin{equation}
\label{eq : new system}
\begin{array}{llcl}
\mbox{\rm (a)} \quad & (1 - c_{1}\,c_{2})\,X^{4}\,\, + \,\, B\,X^{3} \,\,+\,\, C\,X^{2} \,\,+\,\, D\,X \,\,+\,\, b_{1}^{2}\,h_{1}^{2} &=& 0 \\ \\
\mbox{\rm (b)} \quad & (1 - c_{1}\,c_{2})\,Y^{4}\,\, +\,\, \overline{B}\,Y^{3} \,\,+\,\, \overline{C}\,Y^{2} \,\,+\,\, \overline{D}\,Y \,\,+\,\, b_{2}^{2}\,h_{2}^{2} &=& 0
\end{array}
\end{equation}
where $B,\, C,\, D, \,\overline{B},\, \overline{C}$ and $ \overline{D} $ are functions of $b_{1},\,b_{2},\,c_{1},\,c_{2},\,d_{1},\,d_{2}$. 
We consider two cases.

Case 1 : Suppose $c_{1}\,c_{2} < 1 $.  
Since the leading coefficient of (\ref{eq : new system}a) is 
$1 - c_{1}\,c_{2}$ and the constant coefficient is $b_{1}^{2}\,h_{1}^{2}$, 
the product of the roots of (\ref{eq : new system}a) is
$ \displaystyle\frac{b_{1}^{2}\,h_{1}^{2}}{1 - c_{1}\,c_{2}}  > 0$. 
Similarly, the product of the roots of (\ref{eq : new system}b) is 
$ \displaystyle\frac{b_{2}^{2}\,h_{2}^{2}}{1 - c_{1}\,c_{2}} > 0$. 
From Lemma \ref{lem : fixed pts. in Quads. I and III},  
 each equation in \eqref{eq : new system} must have a positive root $\alpha$
and a negative root $\beta$. Furthermore, none of the two equations can have complex roots.
Indeed if, say, (\ref{eq : new system}a) had a pair of complex conjugate roots  
$z$ and $\overline{z}$, then the product of the roots of (\ref{eq : new system}a) given by $\alpha \beta |z|^2$ would 
 be negative,  which is impossible. 
It follows that the set of roots of each equation in  \eqref{eq : new system} 
must have the sign structure $\{+,\,+,\,-,\,-\}$. As a result, the set of equilibria  of 
 Eq.(LGIN) must have one of the sign structures (i.) $ 
\{(+,\,+),\,(-,\,-),\,(+,\,+),\,(-,\,-)\}$ or (ii.) $\{(+,\,+),\,(-,\,-),\,(+,\,-),\,(-,\,+)\}$.
The sign structure given by (i.) is not possible since otherwise, 
there must exist two equilibria of  Eq.(LGIN) 
in $\mathbb{R}^{2}_{+}$ and two equilibria outside $\mathbb{R}^{2}_{+}$. 
By the proof of Theorem \ref{th: main result 1} part iv., 
one of the equilibria in $\mathbb{R}^{2}_{+}$ must have multiplicity two 
while the other equilibrium in $\mathbb{R}^{2}_{+}$ must have multiplicity one. 
Thus the sum total of the multiplicities of all the equilibrium points must be at least five, 
contradicting B{\'e}zout's Theorem.
Hence when $c_{1}\,c_{2} < 1 $, the only possible sign structure for the set of 
equilibria of  Eq.(LGIN) is (ii.), which implies that Eq.(\ref{eq : new system}) has a unique equilibrium
in $Q_1(0,0)$, and hence Eq.(LGIN) has a unique equilibrium in $Q_1(h_1,h_2)$.

Case 2 :  Suppose $c_{1}\,c_{2} =1 $ : In this case, the equations in \eqref{eq : new system} reduce to 
\begin{equation}
\label{eq : new system 2}
\begin{array}{ccc}
(1 - b_{1} - c_{1} + b_{2}\,c_{1})\,X^{3}  +  C_{1}\,X^{2}  +  D_{1}\,X  +  b_{1}^{2}\,h_{1}^{2} &=& 0
 \\ \\
- \,c_{1}\,(1 - b_{1} - c_{1} +  b_{2}\,c_{1})\,Y^{3}  + \overline{C_{1}}\,Y^{2}  +  \overline{ D_{1}}\,Y \,\,+ \,\, b_{2}^{2}\,h_{2}^{2}\,c_{1}^{2} &=& 0
\end{array}
\end{equation}
where $ C_{1},\, D_{1},\, \overline{C}_{1}$ and $ \overline{D}_{1} $ are functions of $b_{1},\,c_{1},\,h_{1},\,b_{2},\,c_{2},\,h_{2}$.
We consider three cases: $1 - b_{1} - c_{1} + b_{2}\,c_{1}$ is negative, zero, and positive.
 If $1 - b_{1} - c_{1} + b_{2}\,c_{1} \neq 0$, the product of the roots of the first equation 
 in \eqref{eq : new system 2} is  $-\displaystyle\frac{b_{1}^{2}\,h_{1}^{2}}{1 - b_{1} - c_{1} + b_{2}\,c_{1}}$ 
 and the product of the roots of the second equation 
 in \eqref{eq : new system 2} is $\displaystyle\frac{b_{2}^{2}\,h_{2}^{2}\,c_{1}}{1 - b_{1} - c_{1} + b_{2}\,c_{1}} $.
Note that  from  Lemma \ref{lem : fixed pts. in Quads. I and III}, we have that each equation in \eqref{eq : new system 2} must have a positive root and a negative root. Thus the equations in \eqref{eq : new system 2} cannot have complex roots due to their cubic nature, since complex roots of real polynomials always occur in conjugate pairs. 
If $1 - b_{1} - c_{1} + b_{2}\,c_{1} < 0$, the sets of roots of the equations in  \eqref{eq : new system 2} must have  sign structures $\{+,\,-,\,-\}$ and $\{+,\,+,\,-\}$, respectively, by an argument similar to the one already used in case 1. Hence the set of equilibria  of  Eq.(LGIN) must have the  sign structure $\{(+,\,+),\,(-,\,-),\,(-,\,+)\}$. If $1 - b_{1} - c_{1} + b_{2}\,c_{1} > 0$, the sets of roots of the equations in  \eqref{eq : new system 2} must have  sign structures $\{+,\,+,\,-\}$ and $\{+,\,-,\,-\}$, respectively. Hence the set of equilibria  of Eq.(LGIN) must have the  sign structure $\{(+,\,+),\,(-,\,-),\,(+,\,-)\}$. If $1 - b_{1} - c_{1} + b_{2}\,c_{1} = 0$, the equations in  \eqref{eq : new system 2} reduce to a pair of quadratic equations. It follows from  Lemma \ref{lem : fixed pts. in Quads. I and III}  that each equation  must have a positive root and a negative root. So the set of roots of each equation must have  the sign structure $\{+,\,-\}$. Hence the set of equilibria  of  Eq.(LGIN) must have the  sign structure $\{(+,\,+),\,(-,\,-)\}$, which implies that Eq.(\ref{eq : new system}) has a unique equilibrium
in $Q_1(0,0)$, and hence Eq.(LGIN) has a unique equilibrium in $Q_1(h_1,h_2)$.
%
%
 \end{itemize}
\end{proof}



\end{document}